\documentclass[11pt,twoside,reqno]{amsart}

\usepackage{microtype}
\usepackage[OT1]{fontenc}
\usepackage{type1cm}
\usepackage{amssymb}
\usepackage{enumerate}
\usepackage{comment}
\usepackage{xcolor}
\usepackage{tikz}

\usepackage{geometry}
\numberwithin{equation}{section}

\theoremstyle{plain}
\newtheorem{theorem}{Theorem}[section]

\newtheorem{proposition}[theorem]{Proposition}

\newtheorem{lemma}[theorem]{Lemma}

\theoremstyle{remark}

\newtheorem{example}[theorem]{Example}
\newtheorem*{ack}{Acknowledgement}

\theoremstyle{definition}

\newcommand{\BB}{\mathcal{B}}
\newcommand{\HH}{\mathcal{H}}

\newcommand{\PP}{\mathcal{P}}

\newcommand{\R}{\mathbb{R}}

\newcommand{\N}{\mathbb{N}}

\newcommand{\iii}{\mathtt{i}}
\newcommand{\jjj}{\mathtt{j}}

\newcommand{\fii}{\varphi}
\newcommand{\roo}{\varrho}

\renewcommand{\ge}{\geqslant}
\renewcommand{\le}{\leqslant}
\renewcommand{\geq}{\geqslant}
\renewcommand{\leq}{\leqslant}

\DeclareMathOperator{\dimh}{dim_H}

\DeclareMathOperator{\dimp}{dim_p}
\DeclareMathOperator{\udimp}{\overline{dim}_p}
\DeclareMathOperator{\ldimp}{\underline{dim}_p}

\DeclareMathOperator{\dimm}{dim_M}
\DeclareMathOperator{\udimm}{\overline{dim}_M}
\DeclareMathOperator{\ldimm}{\underline{dim}_M}

\DeclareMathOperator{\dimf}{dim_F}

\DeclareMathOperator{\dimd}{dim_{\Theta}}

\newcommand{\diml}[1][]{\operatorname{dim_L^{\,#1}}}

\newcommand{\dima}[1][]{\operatorname{dim_A^{\,#1}}}

\newcommand{\udimreg}[1][]{\operatorname{dim_{A}^{\,#1}}}

\DeclareMathOperator{\dimsim}{dim_{sim}}

\DeclareMathOperator{\dist}{dist}
\DeclareMathOperator{\diam}{diam}
\DeclareMathOperator{\diag}{diag}

\begin{document}

\title{Minkowski dimension for measures}

\author{Kenneth J. Falconer}
\address[Kenneth J. Falconer]
        {School of Mathematics and Statistics \\
         University of St Andrews \\
             KY16 9SS \\
         UK}
\email{kjf@st-andrews.ac.uk}

\author{Jonathan M.\ Fraser}
\address[Jonathan M.\ Fraser]
        {School of Mathematics and Statistics \\
         University of St Andrews \\
                 KY16 9SS \\
         UK}
\email{jmf32@st-andrews.ac.uk}

\author{Antti K\"aenm\"aki}
\address[Antti K\"aenm\"aki]
        {Research Unit of Mathematical Sciences \\ 
         P.O.\ Box 8000 \\ 
         FI-90014 University of Oulu \\ 
         Finland}
\email{antti.kaenmaki@oulu.fi}

\thanks{JMF and KJF are financially supported by an  EPSRC Standard Grant (EP/R015104/1) and JMF by a Leverhulme Trust Research Project Grant (RPG-2019-034).}
\date{\today}

\begin{abstract}
The purpose of this article is to introduce and motivate the notion of Minkowski (or box) dimension for measures. The definition is simple and fills a gap in the existing literature on the dimension theory of measures.  As the terminology suggests, we show that it can be used to characterise the Minkowski dimension of a compact metric space.  We also study its relationship with other concepts in dimension theory.
\end{abstract}

\maketitle

\section{Introduction}

It is well-known that the Hausdorff and packing dimensions of a compact metric space $X$ can be approximated arbitrary well from below by the Hausdorff and packing dimensions of measures supported on $X$; see e.g.\ \cite[\S 10]{Falconer1997}. We prove an analogous result for the Minkowski (or box)  dimension.  This first involves introducing upper and lower Minkowski dimensions for measures, and then proving that the Minkowski dimensions of $X$ can be approximated arbitrary well from above, and indeed are attained by, the Minkowski dimensions of measures fully supported on $X$. As working with measures is a rather standard approach in determining the Hausdorff or packing dimension of sets, we expect our new notion to become a useful concept in fractal geometry. Indeed, since the first version of this paper was available online it has already found use in \cite{BaranyJurgaKolossvary2021} where the authors studied the convergence rate of the chaos game. Perhaps most interestingly, it is shown in \cite{AlvaradoGorkaHajlasz2020} that the Minkowski dimension characterizes the existence of Sobolev embeddings. Moreover, our conclusions on Minkowski dimension led us to consider the Frostman dimension and the Assouad spectrum of measures in the last two sections. This has also already found use in \cite{KleptsynPollicottVytnova2022} and \cite{HareHare2020}, respectively.

The upper Minkowski dimension of $\mu$ is defined to be the infimum of all  $s \ge 0$ for which there is a constant $c>0$ such that $\mu(B(x,r)) \ge cr^s$ for all $x \in X$ and $0<r<1$.  We show that the upper Minkowski dimension of a compact set $X$ is the minimum of the upper Minkowski dimensions of measures supported on $X$.  Recall that the Hausdorff dimension of an analytic set $X$ is, by Frostman's lemma,  the supremum of all $s \ge 0$ for which there exists a measure $\mu$ supported on $X$ satisfying $\mu(B(x,r)) \le Cr^s$ for all $x \in X$ and $r>0$ for another  constant $C>0$ independent of $x$ and $r$. Therefore, interestingly, the natural pair with symmetric properties is the Hausdorff dimension and upper Minkowski dimension (of sets).  This is perhaps surprising because it is more often the   Hausdorff and packing dimensions which   behave as a pair.

In order to   motivate our new concept and place it in context,  we study further  properties of the  Minkowski dimensions of measures. So as to present a complete picture, we  fill in some gaps in the literature concerning notions related to the Minkowski dimension including the packing dimension, Assouad spectrum and Frostman dimension.  We show that if $0<r<1$ in the definition of Minkowski dimension  is not assumed to be uniform, then the analogous definition leads to packing dimension. We also show that the upper Minkowski dimension of a  measure is attained as the limiting value of the Assouad spectrum of the measure as the parameter $\theta$ tends to zero.  This is analogous to the situation for sets and further justifies the use of the term Minkowski dimension. The Assouad spectrum is a continuum of dimensions depending on a parameter $0 < \theta < 1$ and is related to the more familiar Assouad dimension.  Finally, we observe that, interestingly, the limiting behaviour of the lower spectrum is different from the Assouad spectrum.

\section{Minkowski dimension}

Let $(X,d)$ be a metric space. Since we use only one metric $d$ on $X$, we simply denote $(X,d)$ by $X$. A closed ball centred at $x\in X$ with radius $r>0$ is denoted by $B(x,r)$. We say that $X$ is \emph{doubling} if there is $N \in \N$ such that any closed ball of radius $r > 0$ can be covered by $N$ balls of radius $r/2$. Furthermore, we call any countable collection $\BB$ of pairwise disjoint closed balls a \emph{packing}. It is called an \emph{$r$-packing} for $r>0$ if all of the balls in $\BB$ have radius $r$. An $r$-packing $\BB$ is termed \emph{maximal} if for every $x \in X$ there is $B \in \BB$ so that $B(x,r) \cap B \ne \emptyset$. Note that if $\BB$ is a maximal $r$-packing, then $2\BB = \{ 2B : B \in \BB \}$ covers $X$. Let $X$ be compact and write
\begin{equation*}
  N_r(X) = \max\{\#\BB : \BB \text{ is an $r$-packing}\} < \infty.
\end{equation*}
The \emph{upper and lower Minkowski dimensions} of $X$ are
\begin{align*}
  \udimm(X) &= \limsup_{r \downarrow 0} \frac{\log N_r(X)}{-\log r}, \\ 
  \ldimm(X) &= \liminf_{r \downarrow 0} \frac{\log N_r(X)}{-\log r},
\end{align*}
respectively. If $\ldimm(X)=\udimm(X)$, then the common value, the \emph{Minkowski dimension} of $X$, is denoted by $\dimm(X)$. Note that equivalent definitions of Minkowski dimensions are given using variants on the definition of $N_r$, see e.g.\ \cite[\S 2.1]{Falconer2014}. Also, the Minkowski dimension is often referred to as the box or box-counting dimension.

The above definitions, and also the definitions of other set dimensions in the coming sections, extend naturally to all subsets of $X$ by considering the restriction metric. Let $\mu$ be a fully supported finite Borel measure on $X$. We define the \emph{upper and lower Minkowski dimensions} of $\mu$ to be
\begin{equation} \label{eq:def-upper-minkowski}
\begin{split}
  \udimm(\mu) = \inf\{ s \ge 0 : \;&\text{there exists a constant } c>0 \text{ such that } \\
  &\mu(B(x,r)) \ge cr^s \text{ for all $x \in X$ and $0<r<1$} \}
\end{split}
\end{equation}
and
\begin{align*}
  \ldimm(\mu) = \inf\{ s \ge 0 : \;&\text{there exist a constant } c>0 \text{ and} \text{ a sequence } (r_n)_{n \in \N} \\ 
  &\text{of positive real numbers such that} \lim_{n\to\infty}r_n=0 \text{ and} \\
  &\mu(B(x,r_n)) \ge cr_n^s \text{ for all $x \in X$ and $n \in \N$} \},
\end{align*}
respectively. In Theorem \ref{thm:minkowski-dim}, we will connect these to the Minkowski dimensions of the support $X$. This connection appears to be rather delicate as not having a uniform $0<r<1$ in \eqref{eq:def-upper-minkowski} leads to packing dimension; see Theorem \ref{thm:packing-dim}. It is easy to see that
\begin{equation*}
  \udimm(\mu) = \limsup_{r \downarrow 0} \sup_{x \in X} \frac{\log\mu(B(x,r))}{\log r}
\end{equation*}
and
\begin{equation*}
  \ldimm(\mu) = \liminf_{r \downarrow 0} \sup_{x \in X} \frac{\log\mu(B(x,r))}{\log r},
\end{equation*}
see \cite[Lemma 1.1]{BaranyJurgaKolossvary2021}. This characterization gives an easy way to compare the Minkowski dimensions to local dimensions of the measure, and therefore also to the Hausdorff and packing dimensions. If $\ldimm(\mu) = \udimm(\mu)$, then the common value, the \emph{Minkowski dimension} of $\mu$, is denoted by $\dimm(\mu)$. Our definitions are different to that of Pesin \cite[\S 7]{Pesin1997}. He introduced quantities which are at most the Minkowski dimension of $X$ whereas ours are at least. As the following theorem shows, the Minkowski dimension of a set can be recovered from the Minkowski dimension of measures supported on the set,  that is, there is a variational principle.  

\begin{theorem} \label{thm:minkowski-dim}
  If $X$ is a compact  metric space, then
  \begin{align*}
    \udimm(X) &= \min\{ \udimm(\mu) : \mu \text{ is a fully supported finite Borel measure on } X \}, \\ 
    \ldimm(X) &= \min\{ \ldimm(\mu) : \mu \text{ is a fully supported finite Borel measure on } X \}.
  \end{align*}
\end{theorem}

\begin{proof}
  Let us first consider the claim for the upper Minkowski dimension.  Let $\mu$ be a fully supported finite Borel measure on $X$ and suppose $\udimm(\mu) < s < \infty$. It follows that there exists a constant $c>0$ such that $\mu(B(x,r)) \ge cr^s$ for all $x \in X$ and $0<r<1$. If $\{B_i\}_{i=1}^N$ is an $r$-packing, then
  \begin{equation*}
    Ncr^s \le \sum_{i=1}^N \mu(B_i) \le \mu(X).
  \end{equation*}
  Since this holds for every $r$-packing, we see that $N_r(X) \le c^{-1}\mu(X)r^{-s}$ for all $0<r<1$ and hence, $\udimm(X) \le s$. This proves one direction of the desired result.

 To show the other direction, we may assume $\udimm(X)<\infty$ since otherwise there is nothing to prove.  Let $k \in \N$ and choose a $2^{-k}$-packing $\BB_k$ such that $N_{2^{-k}}(X) = \#\BB_k$. Note that, by the definition of $N_{2^{-k}}(X)$, $\BB_k$ is maximal and hence, $2\BB_k$ covers $X$. Write $N_k = N_{2^{-k}}(X)$ and $\{B(x_{k,i},2^{-k})\}_{i=1}^{N_k} = \BB_k$. Fix $s>\udimm(X)$, choose $C \ge 1$ such that $N_k \le Ck^{-2} 2^{ks}$ for all $k \in \N$, and define
  \begin{equation*}
    \mu = \sum_{k \in \N} k^{-2} \sum_{i=1}^{N_k} N_k^{-1} \delta_{x_{k,i}},
  \end{equation*}
  where $\delta_x$ is the Dirac measure at $x$. Since
  \begin{equation*}
    \mu(X) = \sum_{k \in \N} k^{-2} \sum_{i=1}^{N_k} N_k^{-1} = \sum_{k \in \N} k^{-2} < \infty,
  \end{equation*}
  $\mu$ is a fully supported finite Borel measure on $X$. Given $x \in X$ and $0<r<1$, choose $k \in \N$ such that $2^{-k+1} < r \le 2^{-k+2}$. Since $\{B(x_{k,i},2 \cdot 2^{-k})\}_{i=1}^{N_k}$ covers $X$, there exists $i \in \{1,\ldots,N_k\}$ such that $x_{k,i} \in B(x,2 \cdot 2^{-k}) \subset B(x,r)$. Therefore,
  \begin{equation*}
    \mu(B(x,r)) \ge k^{-2} N_k^{-1}  \ge  C^{-1} 2^{-ks}
  \end{equation*}
  which proves   $\udimm(\mu) \le s$.  Since $s>\udimm(X)$ was arbitrary, it follows that  $\udimm(\mu)  = \udimm(X)$, completing the proof.

  The claim for the lower Minkowski dimension is proved similarly. To see that $\ldimm(X) \le \ldimm(\mu)$ for all fully supported finite Borel measures $\mu$, just replace arbitrary radii $0<r<1$ by the appropriate sequence $(r_n)_{n\in\N}$ in the corresponding argument for the upper Minkowski dimension. To see the other direction, let $(k_n)_{n \in \N}$ be a strictly increasing sequence of natural numbers such that $\ldimm(X)=\lim_{n\to\infty} \log N_{2^{-k_n}}(X) / \log (2^{k_n})$. Let $n \in \N$ and choose a $2^{-k_n}$-packing $\BB_n$ such that $N_{2^{-k_n}}(X) = \#\BB_n$. Note that, by the definition of $N_{2^{-k_n}}(X)$, $\BB_n$ is maximal and hence $2\BB_n$ covers $X$. Write $N_n = N_{2^{-k_n}}(X)$ and $\{B(x_{n,i},2^{-k_n})\}_{i=1}^{N_n} = \BB_n$. Fix $s>\ldimm(X)$, choose $C \ge 1$ such that $N_n \le Ck_n^{-2} 2^{k_ns}$ for all $n \in \N$, and define a fully supported finite Borel measure
  \begin{equation*}
    \mu = \sum_{n \in \N} k_n^{-2} \sum_{i=1}^{N_n} N_n^{-1} \delta_{x_{n,i}}.
  \end{equation*}
  Write $r_n = 2 \cdot 2^{-k_n}$ for all $n \in \N$ and notice that, for each $x \in X$ and $n \in \N$, we have
  \begin{equation*}
    \mu(B(x,r_n)) \ge k_n^{-2} N_n^{-1}  \ge  C^{-1} 2^{-k_ns}
  \end{equation*}
  and $\ldimm(\mu)  = \ldimm(X)$ as required.
\end{proof}

Theorem \ref{thm:minkowski-dim} generalizes the result of Tricot \cite[Lemma 4]{Tricot1982} whose proof relies on an argument symmetrical to Frostman's lemma and covers only the upper Minkowski dimension. We also remark that,  Theorem \ref{thm:minkowski-dim} contains most useful information when the Minkowski dimensions are finite.  This holds for any compact doubling metric space, for example.  

Recall that a measure $\mu$ on $X$ is \emph{doubling} if there is a constant $C \ge 1$ such that
\begin{equation*}
  0 < \mu(B(x,2r)) \le C\mu(B(x,r)) < \infty
\end{equation*}
for all $x \in X$ and $0<r<1$. The measures constructed in Theorem \ref{thm:minkowski-dim} are clearly not in general doubling measures.  We show that sometimes this cannot be avoided.  Specifically,  in Proposition \ref{thm:inhomogeneous}, we show that for  a large class of inhomogeneous self-similar sets there does not exist a doubling measure supported on the set  with  upper Minkowski dimension equal to that of the set.  We emphasize that by \emph{inhomogeneous self-similar set}, we do not refer to self-similar sets, but a generalization due to Barnsley and Demko \cite{BarnsleyDemko1985} which incorporate a given `condensation' set

\section{Packing dimension}

The \emph{upper} and \emph{lower packing dimensions} of $\mu$ are
\begin{align*}
  \udimp(\mu) &= \inf\{ \dimp(A) : A\subset X\text{ is a Borel set such that }\mu(X \setminus A) = 0 \}, \\ 
  \ldimp(\mu) &= \inf\{ \dimp(A) : A\subset X\text{ is a Borel set such that }\mu(A) > 0 \},
\end{align*}
respectively, where $\dimp(A)$ is the packing dimension of $A \subset X$;  see  \cite[\S 10.1]{Falconer1997} and  \cite[\S 3.5]{Falconer2014}. 
It is well-known that the packing dimension of $X$ can be approximated arbitrary well from below by upper and lower packing dimensions of measures; see \cite[Proposition 10.1]{Falconer1997}. Since the question whether the suprema can be attained here does not seem to be so well documented, we present the full details in the following.

\begin{theorem} \label{thm:packing-dim-below}
  If $X$ is an analytic subset of a metric space, then
  \begin{align*}
    \dimp(X) &= \max\{\udimp(\mu) : \mu\text{ is a finite Borel measure on }X\} \\
    &= \sup\{\ldimp(\mu) : \mu\text{ is a finite Borel measure on }X\}.
  \end{align*}
\end{theorem}

\begin{proof}
  Write $s_n = \dimp(X)-\tfrac{1}{n}$ for all $n \in \N$. For every $n \in \N$, by the result of Joyce and Preiss \cite[Theorem 1]{JoycePreiss1995}, there exists a compact set $K_n \subset X$ such that $0<\PP^{s_n}(K_n)<\infty$, where $\PP^s$ is $s$-dimensional packing measure; see Cutler \cite{Cutler1995b} or \cite[\S 3.5]{Falconer2014} for the definition. Define
  \begin{equation*}
    \mu_n = \frac{\PP^{s_n}|_{K_n}}{\PP^{s_n}(K_n)} \quad\text{and}\quad \mu=\sum_{n \in \N}2^{-n}\mu_n,
  \end{equation*}
  and note that $\mu$ is a Borel probability measure.

  To show the first equality, let $A \subset X$ be a Borel set with $\mu(X \setminus A)=0$. Since $1=\mu(A)=\sum_{n\in\N}2^{-n}\mu_n(A)$, we have $\mu_n(A)=1$ and $\PP^{s_n}(K_n \cap A)=\PP^{s_n}(K_n)$ for all $n \in \N$. Therefore, $\PP^{s_n}(A)\ge\PP^{s_n}(K_n\cap A)=\PP^{s_n}(K_n)>0$ and $\dimp(A)\ge s_n = \dimp(X)-\tfrac{1}{n}$ for all $n \in \N$. It follows that $\dimp(A)=\dimp(X)$ and hence, $\udimp(\mu)=\dimp(X)$.

  To see the second equality, fix $n \in \N$ and let $A \subset X$ be a Borel set such that $\mu_n(A)>0$. Since $\PP^{s_n}(A) \ge \PP^{s_n}(K_n\cap A)=\mu_n(A)\PP^{s_n}(K_n)>0$, we have $\dimp(A) \ge s_n = \dimp(X)-\tfrac{1}{n}$ and hence, $\ldimp(\mu_n) \ge \dimp(X)-\tfrac{1}{n}$ giving the claim.
\end{proof}

By relying on the result of Davies \cite{Davies1952}, see also Rogers \cite{Rogers1998}, or Howroyd \cite{Howroyd1995}, it is possible to modify Theorem \ref{thm:packing-dim-below} for the Hausdorff dimension. The following example can also be easily modified for the Hausdorff dimension.

\begin{figure}[t]
  \begin{tikzpicture}[scale=0.068]
    \draw (0,0) -- (0,81) -- (81,81) -- (81,0) -- cycle;

    \filldraw[fill=black!10!white, draw=black] (54,0) rectangle (81,27);
    \filldraw[fill=black!10!white, draw=black] (54,54) rectangle (81,81);
    \filldraw[fill=black!10!white, draw=black] (0,54) rectangle (27,81);

    \filldraw[fill=black!20!white, draw=black] (18,0) rectangle (27,9);
    \filldraw[fill=black!20!white, draw=black] (18,18) rectangle (27,27);
    \filldraw[fill=black!20!white, draw=black] (0,18) rectangle (9,27);

    \filldraw[fill=black!30!white, draw=black] (6,0) rectangle (9,3);
    \filldraw[fill=black!30!white, draw=black] (6,6) rectangle (9,9);
    \filldraw[fill=black!30!white, draw=black] (0,6) rectangle (3,9);

    \filldraw[fill=black!40!white, draw=black] (2,0) rectangle (3,1);
    \filldraw[fill=black!40!white, draw=black] (2,2) rectangle (3,3);
    \filldraw[fill=black!40!white, draw=black] (0,2) rectangle (1,3);

    \filldraw[fill=black!50!white, draw=black] (0.666,0) rectangle (1,0.333);
    \filldraw[fill=black!50!white, draw=black] (0.666,0.666) rectangle (1,1);
    \filldraw[fill=black!50!white, draw=black] (0,0.666) rectangle (0.333,1);

    \filldraw[fill=black!60!white, draw=black] (0,0) rectangle (0.333,0.333);
  \end{tikzpicture}
  \caption{Illustration for the set $X$ in Example \ref{ex:lower-strict}.}
  \label{fig:illustration}
\end{figure}
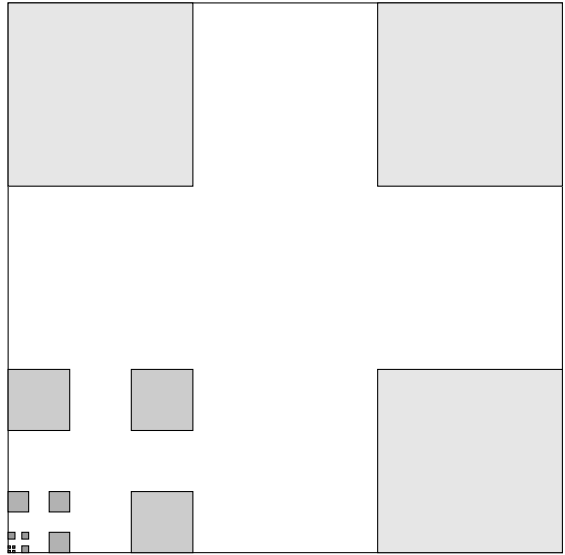

\begin{example} \label{ex:lower-strict}
  In this example, we exhibit a compact set $X \subset \R^2$ for which
  \begin{equation*}
    \dimp(X) > \ldimp(\mu)
  \end{equation*}
  for all finite Borel measures on $X$. Let $0<s\le 2$ and $s_n = s(1-\tfrac{1}{2n}) > 0$ for all $n \in \N$. For each $i \in \{1,\ldots, 4\}$ define a map $\fii_i \colon \R^2 \to \R^2$ by setting
  \begin{equation*}
    \fii_i(x) = \frac{x+t_i}{3},
  \end{equation*}
  where $t_1 = (0,0)$, $t_2 = (0,2)$, $t_3 = (2,2)$, and $t_4 = (2,0)$. Write $\fii_\iii = \fii_{i_1} \circ \cdots \circ \fii_{i_k}$ for all $\iii = i_1\cdots i_k \in \{1,\ldots,4\}^k$ and $k \in \N$. Denote the element $1 \cdots 1$ of $\{1,\ldots,4\}^k$ consisting only of $1$s by $\jjj_k$. Let $X_n \subset \R^2$ be a compact set with $\dimp(X_n) = s_n$ for all $n \in \N$. Define
  \begin{equation*}
    X = \{0\} \cup \bigcup_{k=0}^\infty \bigcup_{i \in \{2,3,4\}} \fii_{\jjj_ki}(X_{k+1});
  \end{equation*}
  see Figure \ref{fig:illustration} for illustration. Observe that $X \subset \R^2$ is compact and, as it contains $s_n$-dimensional subsets, $\dimp(X) \ge s_n$ for all $n \in \N$ and hence $\dimp(X) \ge s$. Let $\mu$ be a finite Borel measure on $X$. If $\mu(X \setminus B(0,r)) = 0$ for all $r>0$, then $\mu$ is supported at the origin and therefore, has dimension zero. But if there is $r>0$ such that $\mu(X \setminus B(0,r)) > 0$, then, by choosing $A = X \setminus B(0,r)$, we have $\mu(A)>0$ and $\dimp(A) \le s_n < s$ for some $n \in \N$. Therefore, $\ldimp(\mu) < s$ as claimed.
\end{example}

Let us next examine whether there exists a result analogous to Theorem \ref{thm:minkowski-dim} for the packing dimension. Define the \emph{lower $s$-density} of $\mu$ at $x \in X$ by
\begin{equation*}
  \Theta^s_*(\mu,x) = \liminf_{r \downarrow 0} \frac{\mu(B(x,r))}{(2r)^s}
\end{equation*}
and notice that, as a function of $s$, it is increasing.

\begin{lemma} \label{thm:antifrostman}
  If $X$ is a compact doubling metric space and $\dimp(X) < s$, then there exists a fully supported finite Borel measure $\mu$ on $X$ such that
  \begin{equation*}
    \Theta^s_*(\mu,x) > 0
  \end{equation*}
  for all $x \in X$.
\end{lemma}

\begin{proof}
  By \cite[\S 5.9]{Mattila1995}, $X$ has a cover $\{X_n\}_{n \in \N}$ of compact sets such that $X_n \subset X$ and $\udimm(X_n) < s$ for all $n \in \N$. Therefore, for each $n \in \N$, by Theorem \ref{thm:minkowski-dim}, there exist a fully supported Borel probability measure $\mu_n$ on $X_n$ and a constant $c_n>0$ such that
  \begin{equation*}
    \mu_n(B(x,r)) \ge c_nr^s
  \end{equation*}
  for all $x \in X_n$ and $0<r<1$. The measure $\mu = \sum_{n \in \N} 2^{-n}\mu_n$ is a fully supported Borel probability measure on $X$ and satisfies
  \begin{equation*}
    \liminf_{r \downarrow 0} \frac{\mu(B(x,r))}{(2r)^s} \ge \liminf_{r \downarrow 0} \frac{\sum_{k \in \{n\in\N:x\in X_n\}}2^{-k}c_kr^s}{(2r)^s} > 0
  \end{equation*}
  for all $x \in X$.
\end{proof}

Define the \emph{density dimension} of $\mu$ to be
\begin{equation*}
  \dimd(\mu) = \inf\{ s \ge 0 : \Theta^s_*(\mu,x)>0 \text{ for all } x\in X \}.
\end{equation*}
Note that $\dimd(\mu) \le \udimm(\mu)$ for all measures $\mu$. The following example shows that the inequality can be strict.

\begin{example} \label{ex:anti-frostman-strict}
  In this example, we exhibit a compact set $X \subset \R$ and a fully supported finite Borel measure $\mu$ on $X$ for which
  \begin{equation*}
    \dimd(\mu) < \udimm(\mu).
  \end{equation*}
  Let $X = \{0\} \cup \{1/n\}_{n \in \N}$ and define
  \begin{equation*}
    \mu = \delta_{0} + \sum_{n=1}^\infty \frac{\delta_{1/n}}{n^2},
  \end{equation*}
  where $\delta_x$ is the Dirac mass at $x$. Notice that $\mu$ is clearly fully supported and $\mu(X) = 1+\sum_{n=1}^\infty n^{-2} = 1+\pi^2/6<\infty$. Therefore, by Theorem \ref{thm:minkowski-dim}, $\udimm(\mu) \ge \udimm(X) = \tfrac12$. 

  Let $s>0$. Fix $n \in \N$ and choose $0 < r < \min\{ \tfrac12(n^2+n)^{-1}, n^{-2/s} \}$. Notice that the ball $B(\tfrac{1}{n}, r)$ contains only the centre point $\tfrac{1}{n}$. Therefore,
  \begin{equation*}
    \mu(B(\tfrac{1}{n}, r)) = \mu(\{\tfrac{1}{n}\}) = n^{-2} \ge r^{s}.
  \end{equation*}
  Since also $\mu(B(0,r)) \ge 1 \ge r^s$, we have shown that $\Theta^s_*(\mu,x) > 0$ for all $x \in X$ and $s > 0$. Therefore, $\dimd(\mu) = 0$.
\end{example}

The following theorem, which generalizes the result of Cutler \cite[Lemma 3.3]{Cutler1995}, is analogous to Theorem \ref{thm:minkowski-dim}. In fact, Theorems \ref{thm:minkowski-dim} and \ref{thm:packing-dim} together show that any set with packing dimension strictly less than upper Minkowski dimension supports finite Borel measures satisfying the property described in Example \ref{ex:anti-frostman-strict}.

\begin{theorem} \label{thm:packing-dim}
  If $X$ is a compact doubling metric space, then
  \begin{equation*}
    \dimp(X) = \min\{\dimd(\mu):\mu\text{ is a fully supported finite Borel measure on }X\}.
  \end{equation*}
\end{theorem}

\begin{proof}
  Let us first show the claim with minimum replaced by infimum.
%
%
  If $\mu$ is a fully supported finite Borel measure on $X$ and $\dimd(\mu) < s$, then $\Theta^s_*(\mu,x)>0$ for all $x \in X$ and thus, by the result of Cutler \cite[Theorem 3.16]{Cutler1995b}, $\dimp(X) \le s$. On the other hand, if $\dimp(X)<s$, then, by Lemma \ref{thm:antifrostman}, there exists a fully supported finite Borel measure $\mu$ on $X$ such that $\Theta^s_*(\mu,x) > 0$ for all $x \in X$ and hence, $\dimd(\mu) \le s$.

  Let us now show that there exists a measure whose density dimension achieves $\dimp(X)$. Write $s_n = \dimp(X) + \tfrac{1}{n}$ for all $n \in \N$. By
  the first part of the proof, there exists a fully supported finite Borel measure on $X$ such that $\dimd(\mu_n) < s_n$ and therefore, $\Theta^{s_n}_*(\mu_n,x)>0$ for all $x \in X$ and $n \in \N$. Define
  \begin{equation*}
    \mu = \sum_{n \in \N} 2^{-n} \mu_n
  \end{equation*}
  and notice that, as in the proof of Lemma \ref{thm:antifrostman}, $\Theta^{s_n}_*(\mu,x)>0$ for all $x \in X$ and $n \in \N$. Hence, $\dimd(\mu) \le s_n = \dimp(X) + \tfrac{1}{n}$ for all $n \in \N$ yielding $\dimd(\mu) = \dimp(X)$ as required.
\end{proof}

We remark that, by relying on the result of Cutler \cite[Lemma 3.5]{Cutler1995}, it is possible to establish an analogue of  Theorem \ref{thm:packing-dim} for Hausdorff dimension.

\section{Assouad spectrum and $L^q$-dimensions}

Recall that if $q \in \R$, then the \emph{$L^q$-spectrum} of $\mu$ is
\begin{equation*}
  \tau_q(\mu) = \liminf_{r \downarrow 0} \frac{\log M_q(\mu,r)}{\log r},
\end{equation*}
where
  $M_q(\mu,r) = \sup\{\sum_{B \in \BB} \mu(B)^q : \BB \text{ is an $r$-packing of } X\}$,
and the \emph{$L^q$-dimension} of $\mu$ is
\begin{equation*}
  \dim_{L^q}(\mu) = \frac{\tau_q(\mu)}{q-1}
\end{equation*}
for $q \neq 1$. It is well known that $\ldimp(\mu) \le \udimp(\mu) \le \dim_{L^q}(\mu)$ for all $-\infty<q<1$; see \cite[Theorem 3.1]{KaenmakiRajalaSuomala2016} and references therein.

Following K\"aenm\"aki, Lehrb\"ack, and Vuorinen \cite{KaenmakiLehrbackVuorinen2013}, see also \cite{Fraser2021}, we define the \emph{Assouad dimension} of $\mu$ to be
\begin{align*}
  \udimreg(\mu) = \inf\{ s \ge 0 : \;&\text{there exists a constant $c>0$ such that } \\
  &\frac{\mu(B(x,r))}{\mu(B(x,R))} \ge c\Bigl( \frac{r}{R} \Bigr)^s \text{ for all $x \in X$ and $0<r<R<1$} \}.
\end{align*}
It is easy to see that $\udimreg(\mu)<\infty$ if and only if $\mu$ is doubling; see \cite[Lemma 3.2]{JJKRRS2010} and \cite[Proposition 3.1]{FraserHowroyd2020}. Although the property this definition captures has been studied earlier (see e.g.\ \cite[\S 13]{Heinonen2001}), the Assouad dimension of measures was explicitly defined in \cite{KaenmakiLehrbackVuorinen2013} where it was called upper regularity dimension. Recall also that, by the result of Fraser and Howroyd \cite[Theorem 2.1]{FraserHowroyd2020}, we have $\dim_{L^q}(\mu) \le \udimreg(\mu)$ for all $-\infty < q < 1$. Following Hare and Troscheit \cite{HareTroscheit2019}, we define the \emph{Assouad spectrum} of the measure $\mu$ by setting
\begin{align*}
  \udimreg[\theta](\mu) = \inf\{ s \ge 0 : \;&\text{there exists a constant $0 < c \le 1$ such that } \\
  &\frac{\mu(B(x,r))}{\mu(B(x,r^\theta))} \ge c\Bigl( \frac{r}{r^\theta} \Bigr)^s \text{ for all $x \in X$ and } 0<r<1 \}
\end{align*}
for all $0<\theta<1$. It follows immediately from the definitions that $\udimreg[\theta](\mu) \le \udimreg(\mu)$ for all $0<\theta<1$ and that $\udimm(\mu) \leq \udimreg(\mu)$. The role of the parameter $\theta$, as the following proposition shows, is to introduce a spectrum of dimensions having values between the upper Minkowski dimension and the Assouad dimension.

\begin{proposition} \label{thm:basic-bounds}
  If $X$ is a compact metric space and $\mu$ is a fully supported finite Borel measure on $X$, then
  \begin{equation*}
    \udimp(\mu) \le \dim_{L^q}(\mu) \le \udimm(\mu) \le \udimreg[\theta](\mu) \le \min\biggl\{ \udimreg(\mu), \frac{\udimm(\mu)}{1-\theta} \biggr\}
  \end{equation*}
  for all $-\infty < q < 1$ and $0<\theta<1$. Moreover, $\udimm(\mu) = \lim_{\theta \downarrow 0} \dima[\theta](\mu)$.
\end{proposition}

\begin{proof}
  Counting from left to right, the first inequality was already stated and referred to above. Let us show the second inequality. Fix $-\infty < q < 1$, choose $\udimm(\mu) < s$, and let $\BB$ be an $r$-packing. Since $\mu(B) \ge cr^{s}$ for all $B \in \BB$ where $c>0$ is a constant,
  \begin{equation*}
    \sum_{B \in \BB} \mu(B)^q = \sum_{B \in \BB} \mu(B)\mu(B)^{q-1} \le c^{q-1}\sum_{B \in \BB} \mu(B)r^{s(q-1)} \le c^{q-1}\mu(X)r^{s(q-1)}.
  \end{equation*}
  This implies $M_q(\mu,r) \le c^{q-1}\mu(X)r^{s(q-1)}$ for all $0<r<1$ and $\tau_q(\mu) \ge s(q-1)$. Hence $\dim_{L^q}(\mu) \le s$ as claimed.

  To show the third inequality, fix $0<\theta<1$ and let $t>s>\udimreg[\theta](\mu)$. This means that there is $0<c<1$ such that
  \begin{equation} \label{eq:basic-bounds1}
    \frac{\mu(B(x,r))}{\mu(B(x,r^\theta))} \ge cr^{(1-\theta)s}
  \end{equation}  
  for all $x \in X$ and $0<r<1$. Since $X$ is compact and $\mu$ is fully supported, there exists $\gamma > 0$ such that
  \begin{equation} \label{eq:basic-bounds2}
    \mu(B(x,\tfrac12)) \ge \gamma
  \end{equation}
  for all $x \in X$. Indeed, if this was not the case, then there would exist a sequence $(x_n)_{n \in \N}$ of points in $X$ such that $\mu(B(x_n,\tfrac12))<\tfrac1n$ for all $n \in \N$. By compactness, $X$ can be covered by finitely many balls of radius $\tfrac14$. If $B$ is one of the covering balls and contains infinitely many points $x_{n_1}, x_{n_2}, \ldots$ from the sequence $(x_n)_{n \in \N}$, then $\mu(B) \le \mu(B(x_{n_i},\tfrac12)) \le \tfrac{1}{n_i}$ for all $i \in \N$ and, consequently, $\mu(B)=0$. This cannot be the case since $\mu$ is fully supported and therefore, the sequence $(x_n)_{n \in \N}$ can contain only finitely many distinct points. But this means that there is a point $x$ appearing infinitely often in the sequence $(x_n)_{n \in \N}$ and therefore, $\mu(B(x,\tfrac12)) = 0$. This contradiction proves \eqref{eq:basic-bounds2}.

  Fix $0<r<1$ and choose $k \in \N$ such that $r^{\theta^{k-1}} < \tfrac12 \le r^{\theta^k}$. This implies
  \begin{equation} \label{eq:basic-bounds3}
    k < \frac{\log(\frac{\log 2}{-\log r})}{\log\theta}+1 \quad \text{and} \quad r^{-\theta^ks} > 2^{\theta s}.
  \end{equation}
  Now, by \eqref{eq:basic-bounds1}, the fact that $r^{\theta^k} \ge \tfrac12$, \eqref{eq:basic-bounds2}, and \eqref{eq:basic-bounds3},
  \begin{align*}
    \mu(B(x,r)) &= \frac{\mu(B(x,r))}{\mu(B(x,r^\theta))} \frac{\mu(B(x,r^\theta))}{\mu(B(x,r^{\theta^2}))} \cdots \frac{\mu(B(x,r^{\theta^{k-1}}))}{\mu(B(x,r^{\theta^k}))} \mu(B(x,r^{\theta^k})) \\
    &\ge c^k r^{(1-\theta)s} r^{\theta(1-\theta)s} \cdots r^{\theta^{k-1}(1-\theta)s} \mu(B(x,\tfrac12)) \\
    &\ge c\Bigl( \frac{\log 2}{-\log r} \Bigr)^{\frac{\log c}{\log \theta}} r^{(1-\theta^{k})s} \gamma \\
    &\ge c\Bigl( \frac{\log 2}{-\log r} \Bigr)^{\frac{\log c}{\log \theta}} \frac{1}{r^{-(s-t)}} r^t 2^{\theta s} \gamma
  \end{align*}
  for all $x \in X$. Since $(-\log r)^{\log c / \log\theta} r^{-(s-t)} \to 0$ as $r \downarrow 0$, it follows that there is a constant $c' > 0$ such that $\mu(B(x,r)) \ge c' r^t$ for all $x \in X$ and $0<r<1$ and hence, $\udimm(\mu) \le t$ as required.

  Let us then show the fourth inequality. Fix $0<\theta<1$ and observe that $\udimreg[\theta](\mu) \le \udimreg(\mu)$ by definition. Therefore, let $s>\udimm(\mu)/(1-\theta)$. This means that there is $c>0$ such that $\mu(B(x,r)) \ge cr^{(1-\theta)s}$ for all $x \in X$ and $0<r<1$. Since now
  \begin{equation*}
    \frac{\mu(B(x,r))}{\mu(B(x,r^\theta))} \ge c\mu(X)^{-1} r^{(1-\theta)s}
  \end{equation*}
  for all $x \in X$ and $0<r<1$, we get $\udimreg[\theta](\mu) \le s$ as required.
  
The final identity follows by letting $\theta\to 0$ in the third and fourth inequalities.
\end{proof}

In fact, the upper Minkowski dimension of a measure can be expressed in terms of the limiting behaviour of the $L^q$-dimensions.

\begin{proposition} \label{thm:basic-boundsnew}
  If $X$ is a compact metric space and $\mu$ is a fully supported finite Borel measure on $X$, then
  \begin{equation*}
    \udimm(\mu) = \sup_{-\infty < q < 1} \dim_{L^q}(\mu) = \lim_{q \to -\infty} \dim_{L^q}(\mu)
\end{equation*}
\end{proposition}

\begin{proof}
  In light of Proposition \ref{thm:basic-bounds} and the fact that  $\dim_{L^q}(\mu)$ is decreasing in $q$, it suffices to prove that $\udimm(\mu) \leq  \lim_{q \to -\infty} \dim_{L^q}(\mu)$.  To this end, let $t< \udimm(\mu)$ and $q<0$.  Therefore, there exist a point $x \in X$ and a sequence $(r_n)_{n \in \N}$ of positive real numbers tending to $0$ such that $\mu(B(x,r_n)) \le r_n^t$ for all $n \in \N$. Since $\{B(x,r_n)\}$ is trivially an $r_n$-packing of $X$, we get
  \begin{equation*}
    M_q(\mu,r_n) \geq \mu(B(x,r_n))^q \geq r_n^{tq}
  \end{equation*}
  for all $n \in \N$ and therefore, $\tau_q(\mu) \leq tq$ and $\dim_{L^q}(\mu)  \geq \frac{tq}{q-1}$. Letting $q \to- \infty$ we see that $\lim_{q \to -\infty} \dim_{L^q}(\mu) \ge t$ which proves the result as the choice of $t< \udimm(\mu)$ was arbitrary.
\end{proof}

Following Assouad \cite{Assouad1983}, we define the \emph{Assouad dimension} of a set $X$ to be
\begin{align*}
  \dima(X) = \inf\{s\ge 0 : \;&\text{there exists a constant $C \ge 1$ such that} \\
  &N_r(B(x,R))\le C\Bigl( \frac{R}{r} \Bigr)^s \text{ for all $x \in X$ and $0<r<R<1$}\}.
\end{align*}
It is easy to see that $\dima(X)<\infty$ if and only if $X$ is doubling. A simple volume argument shows that $\dima(X) \le \udimreg(\mu)$ for all doubling measures $\mu$ on $X$. Vol{$'$}berg and Konyagin \cite[Theorem 4]{VolbergKonyagin1987} have constructed a compact doubling metric space $X$ such that $\dima(X) < \dima(\mu)$ for all fully supported doubling measures $\mu$ on $X$; see also the result of K\"aenm\"aki and Lehrb\"ack \cite[Theorem 5.1]{KaenmakiLehrback2017}. Following Fraser and Yu \cite{FraserYu2016, FraserYu2018}, we define the \emph{Assouad spectrum} of $X$ to be
\begin{align*}
  \dima[\theta](X) &= \limsup_{r \downarrow 0} \frac{\log\sup\{N_r(B(x,r^\theta)):x \in X\}}{(\theta-1)\log r} \\
  &= \inf\{s\ge 0 : \text{there exists a constant $C \ge 1$ such that} \\
  &\qquad\qquad\qquad\; N_r(B(x,r^\theta))\le C\Bigl( \frac{r^\theta}{r} \Bigr)^s \text{ for all $x \in X$ and $0<r<1$}\}
\end{align*}
for all $0<\theta<1$. Recall that, by \cite[Proposition 3.1]{FraserYu2018}, $\udimm(X) = \lim_{\theta \downarrow 0} \dima[\theta](X)$.

\begin{proposition} \label{thm:assouad-spectrum-ineq}
  If $X$ is a doubling metric space and $\mu$ is a fully supported locally finite Borel measure on $X$, then
  \begin{equation*}
    \dima[\theta](X) \le \udimreg[\theta](\mu)
  \end{equation*}
  for all $0<\theta<1$.
\end{proposition}

\begin{proof}
  Fix $0<\theta<1$ and let $s>\udimreg[\theta](\mu)$. Then there is $c>0$ such that
  \begin{equation} \label{eq:assouad-spectrum-ineq1}
    \frac{\mu(B(x,r))}{\mu(B(x,r^\theta))} \ge cr^{(1-\theta)s}
  \end{equation}
  for all $x\in X$ and $0<r<1$. Let $x \in X$, choose $\lambda > 2^{1/\theta}$, and fix
  \begin{equation*}
    0<r<r_\theta=\min\biggl\{\biggl(\frac{\lambda^{\theta-1}}{2}-\frac{1}{\lambda}\biggr)^{1/(1-\theta)}, \frac{1}{\lambda}\biggr\}.
  \end{equation*}
  Let $\{B(x_1,r), \ldots, B(x_P,r)\}$ be a packing of $B(x,r^\theta)$ for some $P \in \N$. Since $X$ is doubling, we only need to consider finite packings. By \cite[Lemma 2.1]{KaenmakiRajalaSuomala2016}, we see that $\{1,\ldots,P\}$ can be partitioned into sets $I_1,\ldots, I_M$, where $M \in \N$ depends only on $X$ and $\lambda$, such that each collection $\{B(x_i,\lambda r)\}_{i \in I_j}$ is a packing of $B(x,r^\theta)$. Since
  \begin{equation*}
    2+2\lambda r^{1-\theta} \le 2+2\lambda\biggl(\frac{\lambda^{\theta-1}}{2}-\frac{1}{\lambda}\biggr) = \lambda^\theta
  \end{equation*}
  and
  \begin{equation*}
    B(x_i,\lambda r) \subset B(x,r^\theta+\lambda r) \subset B(x_i,2r^\theta+2\lambda r) \subset B(x_i,(\lambda r)^\theta)
  \end{equation*}
  for all $i \in \N$, we have, by \eqref{eq:assouad-spectrum-ineq1},
  \begin{equation*}
    1 \ge \sum_{i \in I_j} \frac{\mu(B(x_i,\lambda r))}{\mu(B(x,r^\theta+\lambda r))} \ge \sum_{i \in I_j} \frac{\mu(B(x_i,\lambda r))}{\mu(B(x_i,(\lambda r)^\theta))} \ge \# I_j c(\lambda r)^{(1-\theta)s}
  \end{equation*}
  for all $j \in \{1,\ldots,M\}$. Therefore,
  \begin{equation*}
    P = \sum_{j=1}^M \# I_j \le \frac{M}{c(\lambda r)^{(1-\theta)s}} = \frac{M}{c\lambda^{(1-\theta)s}} \Bigl( \frac{r^\theta}{r} \Bigr)^s
  \end{equation*}
  and
  \begin{equation*}
    N_r(B(x,r^\theta)) \le \frac{M}{c\lambda^{(1-\theta)s}} \Bigl( \frac{r^\theta}{r} \Bigr)^s
  \end{equation*}
  for all $x \in X$ and $0<r<r_\theta$. Hence, $\dima[\theta](X) \le s$ as claimed.
\end{proof}

We consider a tuple $\Phi = (\fii_i)_{i=1}^N$, where $N \ge 2$ is an integer, of contracting mappings acting on $\R^d$. The  \emph{invariant set} associated to $\Phi$ is the unique non-empty compact set $X \subset \R^d$ satisfying
\begin{equation*}
  X = \bigcup_{i=1}^N \fii_i(X);
\end{equation*}
see \cite{Hutchinson1981}. Let us now assume that each of the map $\fii_i$ is a \emph{similitude}, i.e.\ satisfies $|\fii_i(x)-\fii_i(y)| = r_i|x-y|$ for some contraction coefficient $0<r_i<1$. In this case, the corresponding invariant set is called \emph{self-similar}. Furthermore, if $C \subset \R^d$ is compact, then the \emph{inhomogeneous self-similar set} with condensation set $C$ associated to $\Phi$ is the unique non-empty compact set $X_C \subset \R^d$ such that
\begin{equation*}
  X_C = C \cup \bigcup_{i=1}^N \fii_i(X_C) = X \cup \bigcup_{\iii \in \Sigma_*} \fii_\iii(C),
\end{equation*}
where $X$ is the self-similar set associated to $\Phi$, see \cite{Barnsley2006,BarnsleyDemko1985}. Here the set $\Sigma_*$ is the set of all finite words $\bigcup_{n \in \N} \Sigma_n$, where $\Sigma_n = \{1,\ldots,N\}^n$ for all $n \in \N$. If $\iii = i_1\cdots i_n \in \Sigma_n$ for some $n \in \N$, then $\sigma^k(\iii) = i_{k+1}\cdots i_n \in \Sigma_{n-k}$ for all $k \in \{0,\ldots,n-1\}$. The set $\Sigma = \{1,\ldots,N\}^\N$ is the set of all infinite words. If $\iii = i_1i_2\cdots \in \Sigma$, then $\iii|_n = i_1\cdots i_n \in \Sigma_n$ for all $n \in \N$. Finally, if $\iii = i_1\cdots i_n \in \Sigma_n$ for some $n \in \N$, then $\fii_\iii = \fii_{i_1} \circ \cdots \circ \fii_{i_n}$.

We say that $\Phi$ satisfies the \emph{condensation open set condition (COSC)} with condensation set $C$ if there exists an open set $U \subset \R^d$ such that $C \subset U \setminus \bigcup_{i=1}^N \overline{\fii_i(U)}$, $\fii_i(U) \subset U$ for all $i \in \{1,\ldots,N\}$, and $\fii_i(U) \cap \fii_j(U) = \emptyset$ whenever $i \ne j$. Without the reference to the condensation set $C$, this is the familiar open set condition which, by \cite{Hutchinson1981}, implies that
\begin{equation} \label{eq:COSC}
  \dimh(X) = \dima(X) = \dimsim(\Phi),
\end{equation}
where the \emph{similitude dimension} $\dimsim(\Phi)$ is the unique number $s \ge 0$ for which $\sum_{i=1}^N r_i^s = 1$.

The following proposition extends the observation of Vol{$'$}berg and Konyagin \cite[Theorem 4]{VolbergKonyagin1987} to the Assouad spectrum. It also shows that in a large class of inhomogeneous self-similar sets there does not exist a doubling measure which achieves the minimum in Theorem \ref{thm:minkowski-dim}.

\begin{proposition} \label{thm:inhomogeneous}
  Let $C \subset \R^d$ be a non-empty compact set and let $\Phi$ be a tuple of contractive similitudes  satisfying the COSC with condensation set $C$.  Suppose that $\dima(C)<\dimsim(\Phi)$.   Then the inhomogeneous self-similar set $X_C$ satisfies
  \begin{equation*}
    \inf_{0<\theta<1}(\dima[\theta](\mu) - \dima[\theta](X_C)) > 0 \quad \text{and} \quad \udimm(X_C) < \udimm(\mu)
  \end{equation*}
  for all doubling measures $\mu$ fully supported on $X_C$.
\end{proposition}

\begin{proof}
  Observe that, by Proposition \ref{thm:assouad-spectrum-ineq}, $\inf_{0<\theta<1}(\dima[\theta](\mu) - \dima[\theta](X_C)) \ge 0$ for all fully supported finite Borel measures $\mu$. It suffices to show that this infimum is positive for all doubling measures $\mu$ since then, $\udimm(X_C) < \udimm(\mu)$ follows from \cite[Proposition 3.1]{FraserYu2018} and Proposition \ref{thm:basic-bounds}.
  
  Write $s = \dimsim(\Phi)$ and let $\mu$ be a doubling measure on $X_C$. By \cite[Theorem 4.1]{KaenmakiLehrback2017} and \eqref{eq:COSC}, we have $\dima[\theta](X_C) \le \dima(X_C) = \max\{\dima(X), \dima(C)\} = s$ for all $0<\theta<1$. Hence, to show the claim, it is enough to prove that
  \begin{equation} \label{eq:inhomog-goal}
    \inf_{0<\theta<1}\dima[\theta](\mu) > s.
  \end{equation}
  We follow \cite[proof of Theorem 5.1]{KaenmakiLehrback2017} to see that if $x \in C$, then there are $\iii \in \Sigma$, $0<\roo<\dist(X,C)$, and $0<c<1$ such that
  \begin{equation} \label{eq:inhomog-obs}
  \begin{split}
    \mu(B(\fii_{\iii|_n}(x),\roo r_{\iii|_n})) &\le \mu(\fii_{\iii|_n}(X_C)) \le c^{n-m}r_{\sigma^m(\iii|_n)}^s\mu(\fii_{\iii|_m}(X_C)) \\ &\le c^{n-m}r_{\sigma^m(\iii|_n)}^s\mu(B(\fii_{\iii|_n}(x),r_{\iii|_m}\diam(X_C)))
  \end{split}
  \end{equation}
  for all $n \in \N$ and $m \in \{1,\ldots,n\}$. Indeed, the second inequality in \eqref{eq:inhomog-obs} holds since, by \cite[Equation (5.4)]{KaenmakiLehrback2017}, $\mu(\fii_{\iii|_n}(X_C)) \le cr_{\sigma^{n-1}(\iii|_n)}\mu(\fii_{\iii|_{n-1}}(X_C))$. Write $r = \min_{i \in \{1,\ldots,N\}}r_i$ and let $\eta = \tfrac12 \log c/\log r>0$. Note that $cr^{-\eta}<1$ and $r^k\le r_{\jjj}$ for all $\jjj \in \Sigma_k$ and $k \in \N$. For each $0<\theta<1$ and $n \in \N$ choose $m \in \{1,\ldots,n\}$ such that
  \begin{equation} \label{eq:inhomog-choice}
    r_{\iii|_m}\diam(X_C) \le (\roo r_{\iii|_n})^\theta < r_{\iii|_{m-1}}\diam(X_C).
  \end{equation}
  Relying on \eqref{eq:inhomog-obs} and \eqref{eq:inhomog-choice}, we see that
  \begin{equation*}
    \frac{\mu(B(\fii_{\iii|_n}(x),\roo r_{\iii|_n}))}{\mu(B(\fii_{\iii|_n}(x),(\roo r_{\iii|_n})^\theta))} \le c^{n-m}r_{\sigma^m(\iii|_n)}^s \le (cr^{-\eta})^{n-m} r_{\sigma^m(\iii|_n)}^{s+\eta}.
  \end{equation*}
  Since $n-m \to \infty$ and $(cr^{-\eta})^{n-m} \to 0$ as $n \to \infty$, it follows that $\dima[\theta](\mu) \ge s+\eta > s$ for all $0<\theta<1$. Noting that $\eta$ does not depend on $\theta$, we have thus shown \eqref{eq:inhomog-goal} and finished the proof.
\end{proof}

\section{Lower spectrum}

A natural counterpart to the Assouad dimension is the lower dimension introduced by Larman \cite{Larman1967}. Analogously to the Assouad spectrum, the lower dimension gives rise to the lower spectrum. The \emph{lower spectrum} of $\mu$ is
\begin{align*}
  \diml[\theta](\mu) = \sup\{ s \ge 0 : \;&\text{there exists a constant $C \ge 1$ such that} \\
  &\frac{\mu(B(x,r))}{\mu(B(x,r^\theta))} \le C\Bigl( \frac{r}{r^\theta} \Bigr)^s \text{ for all $x \in X$ and } 0<r<1 \}
\end{align*}
for all $0<\theta<1$ and the \emph{Frostman dimension} of $\mu$ is
\begin{align*}
  \dimf(\mu) = \sup\{ s \ge 0 : \;&\text{there exists a constant $C \ge 1$ such that} \\
  &\mu(B(x,r)) \le Cr^s \text{ for all $x \in X$ and } 0<r<1 \}.
\end{align*}

\begin{proposition} \label{thm:frostman-dim}
  If $X$ is a compact metric space and $\mu$ is a fully supported finite Borel measure on $X$, then
  \begin{equation*}
    \diml[\theta](\mu) \le \dimf(\mu)
  \end{equation*}
  for all $0<\theta<1$.
\end{proposition}

\begin{proof}
  Let $0<\theta<1$ and $t<s<\diml[\theta](\mu)$. Fix $0<r<1$ and choose $k \in \N$ such that $r^{\theta^{k-1}} < \tfrac12 \le r^{\theta^k}$. This implies
  \begin{equation*}
    k < \frac{\log(\frac{\log 2}{-\log r})}{\log\theta}+1 \quad \text{and} \quad r^{-\theta^ks} \le 2^{s}.
  \end{equation*}
  Similarly as in the proof of Proposition \ref{thm:basic-bounds}, we see that
  \begin{align*}
    \mu(B(x,r)) &= \frac{\mu(B(x,r))}{\mu(B(x,r^{\theta}))} \frac{\mu(B(x,r^{\theta}))}{\mu(B(x,r^{\theta^2}))} \cdots \frac{\mu(B(x,r^{\theta^{k-1}}))}{\mu(B(x,r^{\theta^k}))} \mu(B(x,r^{\theta^k})) \\ 
    &\le C^k r^{(1-\theta^k)s} \mu(X) 
    \le 2^sC\mu(X) \biggl( \frac{-\log r}{\log 2} \biggr)^{\frac{\log C}{-\log\theta}} r^{s-t} r^t
  \end{align*}
  for all $x \in X$. Since $(-\log r)^{-\log C/\log\theta} r^{s-t} \to 0$ as $r \downarrow 0$, it follows that there is a constant $C' \ge 1$ such that $\mu(B(x,r)) \le C'r^t$ for all $x \in X$ and $0<r<1$. Hence, $\dimf(\mu) \ge t$ as required.
\end{proof}

The \emph{lower spectrum} of $X$ is
\begin{align*}
  \diml[\theta](X) &= \sup\{s\ge 0 : \text{there exists a constant $0 < c \le 1$ such that} \\
  &\qquad\qquad\qquad\; N_r(B(x,r^\theta))\ge c\Bigl( \frac{r^\theta}{r} \Bigr)^s \text{ for all $x \in X$ and $0<r<1$}\}
\end{align*}
for all $0<\theta<1$. Recall that, by Theorem \ref{thm:minkowski-dim}, \cite[Proposition 3.1]{FraserYu2018}, and Proposition \ref{thm:basic-bounds}, 
  $\lim_{\theta \downarrow 0} \dima[\theta](X) = \inf\{ \lim_{\theta \downarrow 0} \dima[\theta](\mu) : \mu$ is a fully supported finite Borel measure on $X \}$.
The following example shows that there is no analogous result for the lower spectrum.

Let $q > p > 1$ and $N \in \{2,\ldots,pq\}$ be integers, and $A \subset \{0,\ldots,p-1\} \times \{0,\ldots,q-1\}$ a set of $N$ elements. A \emph{Bedford-McMullen carpet} is the invariant set $X \subset [0,1]^2$ associated to a tuple $(\fii_i)_{i=1}^N$ of distinct affine mappings which all have the same linear part $\diag(\tfrac{1}{p},\tfrac{1}{q})$ and the translation part is from the set $\{(\tfrac{j}{p},\tfrac{k}{q}) \in [0,1]^2 : (j,k) \in A\}$. Write $n_j = \#\{k : (j,k) \in A\}$ to denote the number of sets $\fii_i([0,1)^2)$ the vertical line $\{(\tfrac{j}{p},y) : y \in \R\}$ intersects. If there is $n \in \N$ such that $n_j = n$ for all $j$ with $n_j \ne 0$, in which case we say the Bedford-McMullen carpet $X$ has \emph{uniform fibers}, then
\begin{equation*}
  \dimh(X) = \dimm(X) = \dima(X);
\end{equation*}
otherwise,
\begin{equation*}
  \dimh(X) < \dimm(X) < \dima(X);
\end{equation*}
see \cite{Mackay2011}. Here $\dimh$ denotes the Hausdorff dimension; see \cite[\S 3.2]{Falconer2014} or \cite[\S 4]{Mattila1995}.

\begin{example} \label{ex:lower-bad}
  In this example, we exhibit a compact set $X \subset \R^2$ for which there exist $\eta>0$ such that
  \begin{equation*}
    \lim_{\theta \downarrow 0} \diml[\theta](X) - \lim_{\theta \downarrow 0} \diml[\theta](\mu) \ge \eta
  \end{equation*}
  for all finite Borel measures $\mu$ on $X$. By the result of Fraser and Yu \cite[Theorem 3.3]{FraserYu2016}, for any Bedford-McMullen carpet $X$ it is the case that $\lim_{\theta \downarrow 0} \diml[\theta](X) = \dimm(X)$. Let $X$ be a Bedford-McMullen carpet such that $\dimh(X)<\dimm(X)$. Write $\eta = (\dimm(X) - \dimh(X))/2 > 0$ and notice that 
  \begin{equation} \label{eq:lower-bad1}
    \lim_{\theta \downarrow 0} \diml[\theta](X) \ge \dimh(X)+\eta.
  \end{equation}
  Let $\mu$ be a finite Borel measure on $X$. If $s < \dimf(\mu)$, then there is a constant $C \ge 1$ such that $\mu(B(x,r)) \le Cr^s$ for all $x \in X$ and $0<r<1$. Since $\mu(X) \le \sum_i \mu(U_i) \le C\sum_i \diam(U_i)^s$ for all $\delta$-covers $\{U_i\}_i$ of $X$, we get $\HH^s_\delta(X) \ge \mu(X) > 0$ for all $\delta > 0$ and, consequently, the $s$-dimensional Hausdorff measure of $X$ is $\HH^s(X)=\lim_{\delta \downarrow 0}\HH^s_\delta(X)>0$. It follows that
  \begin{equation} \label{eq:lower-bad2}
    \dimh(X) \ge \dimf(\mu).
  \end{equation}
  Finally, by \eqref{eq:lower-bad1}, Proposition \ref{thm:frostman-dim}, and \eqref{eq:lower-bad2},
  \begin{equation*}
    \lim_{\theta \downarrow 0} \diml[\theta](X) - \lim_{\theta \downarrow 0} \diml[\theta](\mu) \ge \dimh(X) + \eta - \dimf(\mu) \ge \eta
  \end{equation*}
  as desired.
\end{example}

By the result of Fraser \cite[Theorem 6.3.1]{Fraser2021},  $\lim_{\theta \downarrow 0} \diml[\theta](X) = \ldimm(X)$ for every invariant set $X$ associated to a tuple of bi-Lipschitz contractions. Therefore, any such $X$ satisfying $\dimh(X) < \ldimm(X)$ has the property described in Example \ref{ex:lower-bad}.

By \eqref{eq:lower-bad2} and the Frostman's lemma (see e.g.\ \cite[Theorem 8.8]{Mattila1995}), we have
\begin{equation*}
  \dimh(X) = \sup\{ \dimf(\mu) : \mu \text{ is a finite Borel measure on } X \}.
\end{equation*}
Therefore, recalling Theorem \ref{thm:minkowski-dim}, the natural pair with symmetric properties is Hausdorff dimension and upper Minkowski dimension. This is interesting as usually Hausdorff and packing dimensions (or measures) form the natural pair.

\begin{ack}
  The authors thank V.\ Suomala for finding a gap in one of the proofs in the first version of the paper. The authors also thank P. Mattila for useful discussion and J. Bj\"orn for pointing out the reference \cite{AlvaradoGorkaHajlasz2020}.
\end{ack}


\begin{thebibliography}{10}

\bibitem{AlvaradoGorkaHajlasz2020}
R.~Alvarado, P.~G{\'{o}}rka, and P.~Haj{\l }asz.
\newblock Sobolev embedding for {$M^{1, p}$} spaces is equivalent to a lower
  bound of the measure.
\newblock {\em J. Funct. Anal.}, 279(7):108628, 39, 2020.

\bibitem{Assouad1983}
P.~Assouad.
\newblock Plongements lipschitziens dans {${\bf R}^{n}$}.
\newblock {\em Bull. Soc. Math. France}, 111(4):429--448, 1983.

\bibitem{BaranyJurgaKolossvary2021}
B.~B{\'a}r{\'a}ny, N.~Jurga, and I.~Kolossv{\'a}ry.
\newblock On the convergence rate of the chaos game.
\newblock {\em Int. Math. Res. Not. IMRN}.
\newblock To appear, available at arXiv:2102.02047.

\bibitem{Barnsley2006}
M.~F. Barnsley.
\newblock {\em Superfractals}.
\newblock Cambridge University Press, 2006.

\bibitem{BarnsleyDemko1985}
M.~F. Barnsley and S.~Demko.
\newblock Iterated function systems and the global construction of fractals.
\newblock {\em Proc. Roy. Soc. London Ser. A}, 399(1817):243--275, 1985.

\bibitem{Cutler1995b}
C.~D. Cutler.
\newblock The density theorem and {H}ausdorff inequality for packing measure in
  general metric spaces.
\newblock {\em Illinois J. Math.}, 39(4):676--694, 1995.

\bibitem{Cutler1995}
C.~D. Cutler.
\newblock Strong and weak duality principles for fractal dimension in
  {E}uclidean space.
\newblock {\em Math. Proc. Cambridge Philos. Soc.}, 118(3):393--410, 1995.

\bibitem{Davies1952}
R.~O. Davies.
\newblock Subsets of finite measure in analytic sets.
\newblock {\em Indagationes Math.}, 14:488--489, 1952.

\bibitem{Falconer1997}
K.~J. Falconer.
\newblock {\em Techniques in fractal geometry}.
\newblock John Wiley \& Sons Ltd., Chichester, 1997.

\bibitem{Falconer2014}
K.~J. Falconer.
\newblock {\em Fractal geometry}.
\newblock John Wiley \& Sons, Ltd., Chichester, third edition, 2014.
\newblock Mathematical foundations and applications.

\bibitem{Fraser2021}
J.~M. Fraser.
\newblock {\em Assouad Dimension in Fractal Geometry: theory, variations, and
  applications}.
\newblock Cambridge University Press, 222, 2020.

\bibitem{FraserHowroyd2020}
J.~M. Fraser and D.~C. Howroyd.
\newblock On the upper regularity dimensions of measures.
\newblock {\em Indiana Univ. Math. J.}, 69(2):685--712, 2020.

\bibitem{FraserYu2016}
J.~M. Fraser and H.~Yu.
\newblock Assouad-type spectra for some fractal families.
\newblock {\em Indiana Univ. Math. J.}, 67(5):2005--2043, 2018.

\bibitem{FraserYu2018}
J.~M. Fraser and H.~Yu.
\newblock New dimension spectra: finer information on scaling and homogeneity.
\newblock {\em Adv. Math.}, 329:273--328, 2018.

\bibitem{HareHare2020}
K.~E. Hare and K.~G. Hare.
\newblock Intermediate {A}ssouad-like dimensions for measures.
\newblock {\em Fractals}, 28(07):2050143, 2020.

\bibitem{HareTroscheit2019}
K.~E. Hare and S.~Troscheit.
\newblock Lower assouad dimension of measures and regularity.
\newblock {\em Math. Proc. Cambridge Philos. Soc.}, 170(2):379--416, 2021.


\bibitem{Heinonen2001}
J.~Heinonen.
\newblock {\em Lectures on analysis on metric spaces}.
\newblock Universitext. Springer-Verlag, New York, 2001.

\bibitem{Howroyd1995}
J.~D. Howroyd.
\newblock On dimension and on the existence of sets of finite positive
  {H}ausdorff measure.
\newblock {\em Proc. London Math. Soc. (3)}, 70(3):581--604, 1995.

\bibitem{Hutchinson1981}
J.~E. Hutchinson.
\newblock Fractals and self-similarity.
\newblock {\em Indiana Univ. Math. J.}, 30(5):713--747, 1981.

\bibitem{JJKRRS2010}
E.~J\"arvenp\"a\"a, M.~J\"arvenp\"a\"a, A.~K\"aenm\"aki, T.~Rajala, S.~Rogovin,
  and V.~Suomala.
\newblock Packing dimension and {A}hlfors regularity of porous sets in metric
  spaces.
\newblock {\em Math. Z.}, 266(1):83--105, 2010.

\bibitem{JoycePreiss1995}
H.~Joyce and D.~Preiss.
\newblock On the existence of subsets of finite positive packing measure.
\newblock {\em Mathematika}, 42(1):15--24, 1995.

\bibitem{KaenmakiLehrback2017}
A.~K\"{a}enm\"{a}ki and J.~Lehrb\"{a}ck.
\newblock Measures with predetermined regularity and inhomogeneous self-similar
  sets.
\newblock {\em Ark. Mat.}, 55(1):165--184, 2017.

\bibitem{KaenmakiLehrbackVuorinen2013}
A.~K\"aenm\"aki, J.~Lehrb\"ack, and M.~Vuorinen.
\newblock Dimensions, {W}hitney covers, and tubular neighborhoods.
\newblock {\em Indiana Univ. Math. J.}, 62(6):1861--1889, 2013.

\bibitem{KaenmakiRajalaSuomala2016}
A.~K\"aenm\"aki, T.~Rajala, and V.~Suomala.
\newblock Local homogeneity and dimensions of measures.
\newblock {\em Ann. Sc. Norm. Super. Pisa Cl. Sci. (5)}, 16(4):1315--1351,
  2016.

\bibitem{KleptsynPollicottVytnova2022}
V.~Kleptsyn, M.~Pollicott, and P.~Vytnova.
\newblock Uniform lower bounds on the dimension of {B}ernoulli convolutions.
\newblock {\em Adv. Math.}, 395:Paper No. 108090, 2022.

\bibitem{Larman1967}
D.~G. Larman.
\newblock A new theory of dimension.
\newblock {\em Proc. London Math. Soc. (3)}, 17:178--192, 1967.

\bibitem{Mackay2011}
J.~M. Mackay.
\newblock Assouad dimension of self-affine carpets.
\newblock {\em Conform. Geom. Dyn.}, 15:177--187, 2011.

\bibitem{Mattila1995}
P.~Mattila.
\newblock {\em Geometry of sets and measures in Euclidean spaces: fractals and
  rectifiability}.
\newblock Cambridge University Press, 1995.

\bibitem{Pesin1997}
Y.~B. Pesin.
\newblock {\em Dimension theory in dynamical systems}.
\newblock Chicago Lectures in Mathematics. University of Chicago Press,
  Chicago, IL, 1997.
\newblock Contemporary views and applications.

\bibitem{Rogers1998}
C.~A. Rogers.
\newblock {\em Hausdorff measures}.
\newblock Cambridge Mathematical Library. Cambridge University Press,
  Cambridge, 1998.

\bibitem{Tricot1982}
C.~Tricot.
\newblock Two definitions of fractional dimension.
\newblock {\em Math. Proc. Cambridge Philos. Soc.}, 91(1):57--74, 1982.

\bibitem{VolbergKonyagin1987}
A.~L. Vol{$'$}berg and S.~V. Konyagin.
\newblock On measures with the doubling condition.
\newblock {\em Izv. Akad. Nauk SSSR Ser. Mat.}, 51(3):666--675, 1987.

\end{thebibliography}

\end{document}